\documentclass[reqno,a4paper]{amsart}
\usepackage[active]{srcltx}
\usepackage{fancyhdr}
\usepackage[utf8]{inputenc}
\DeclareRobustCommand{\SkipTocEntry}[5]{} 
\usepackage{times}
\usepackage{latexsym,amsopn,amssymb,amsmath,mathabx}
\changenotsign
\usepackage{amsthm}
\usepackage{amsfonts,amsbsy,amscd,stmaryrd,bbold}
\usepackage{dsfont}
\usepackage{paralist}
\usepackage{hyperref}
\hypersetup{colorlinks=true,linkcolor=Emerald,urlcolor=blue} 

\usepackage[usenames,dvipsnames]{color}

\usepackage{mathrsfs} 

\usepackage{enumitem}

\newcommand{\C}{\mathbb{C}}
\newcommand{\N}{\mathbb{N}}
\newcommand{\R}{\mathbb{R}}

\newcommand{\ca}{\mathcal{A}}
\newcommand{\cb}{\mathcal{B}}

\newcommand{\cd}{\mathcal{D}}

\newcommand{\cf}{\mathcal{F}}
\newcommand{\cg}{\mathcal{G}}
\newcommand{\ch}{\mathcal{H}}

\newcommand{\cs}{\mathcal{S}}

\def\rond{\mathscr}

\newcommand{\rb}{\rond{B}}

\def\rmc{\mathrm{c}}

\def\d{\mathrm{d}}

\def\rme{\mathrm{e}}

\def\rmi{\mathrm{i}}
\def\i{\mathrm{i}}

\def\jap#1{\langle {#1} \rangle}

\def\qed{\hfill $\Box$\medskip}

\newtheorem{theorem}{Theorem}[section]
\newtheorem{lemma}[theorem]{Lemma}
\newtheorem{proposition}[theorem]{Proposition}

\newtheorem{definition}[theorem]{Definition}

\numberwithin{equation}{section}

\begin{document}

\title[LAP at zero energy]{On the Limiting absorption principle at zero energy for a new class of possibly non self-adjoint Schr\"odinger operators 
  \\[2mm]
  {\tiny \today}
}

\author[A. Martin]{Alexandre Martin} \address{A. Martin,
  D\'epartement de Math\'ematiques, Universit\'e de Cergy-Pontoise,
  95000 Cergy-Pontoise, France}
\email{alexandre.martin@u-cergy.fr }

\begin{abstract}
We recall a Moure theory adapted to non self-adjoint operators and we apply this theory to Schr\"odinger operators with non real potentials, using different type of conjugate operators. We show that some conjugate operators permits to relax conditions on the derivatives of the potential that were required up to now.
\end{abstract}

\maketitle
\tableofcontents

\section{Introduction}

In this article, we will study the Schr\"odinger operator $H=\Delta+V$ with possibly a non-real potential, on $L^2(\R^n)$, where $\Delta$ is the non negative Laplacian operator. Here $V$ is a multiplication operator, i.e. $V$ can be the operator of multiplication by a function or by a distribution of strictly positive order. When $V=0$, we know that $H=\Delta$ on $L^2(\R^n)$ has for spectrum the real set $[0,+\infty)$ with purely absolutely continuous spectrum on this set. In this article, we are always in the application framework of the Weyl's Theorem. In particular, the essential spectrum of $H$ is the same that the essential spectrum of $\Delta$, the interval $[0,+\infty)$. Thus, $0$ is the bound of the essential spectrum and, for this reason, $0$ is a threshold for $H$. Here, we are interested in the nature of the essential spectrum of the perturbed operator and in the existence of a Limiting Absorption Principle near the threshold $0$.

A general technique to prove a Limiting Absorption Principle if $H$ is self-adjoint is due to E. Mourre \cite{Mo1} and it involves a local version of the positive commutator method due to C.R. Putnam \cite{Pu1,Pu2}. This method is based on the research of another self-adjoint operator $A$, named the conjugate operator, for which the operator $H$ is "regular" with respect to $A$ and for which the Mourre estimate is satisfied on a set $I$ in the following sense:
\[E(I)[H,iA]E(I)\geq c_0E(I)+K,\]
where $c_0>0$, $E$ is the spectral mesure of $H$ and $K$ is a compact operator. When $H$ is a Schr\"odinger operator, we usually apply the Mourre theory with the generator of dilations $A_D$ as conjugate operator. With this conjugate operator, we obtain for the first order commutator of the Laplacian $[\Delta,iA_D]=2\Delta$. In particular, by considering potential such that $H$ is a compact perturbation of the Laplacian, and under some assumptions on it, we can prove the Mourre estimate if $I$ is a compact interval of $(0,+\infty)$. This implies a Limiting Absorption Principle on all compact interval of $(0,+\infty)$ (see \cite{ABG}). But, we can see that since $E(I)[\Delta,iA_D]E(I)$ is not strictly positive when $0\in I$, we can not use the Mourre theorem to prove a Limiting Absorption Principle at zero energy. To do this, several methods linked to Mourre theory exist. A first method uses the standard Mourre theory with a parameter. The goal is to obtain a Limiting Absorption Principle for a modified operator which depends on the parameter and to deduce from this Limiting Absorption Principle a similar estimate for the initial operator, without the parameter (see \cite{BHa}). A second method is to show a Limiting Absorption Principle with weights which depends on a parameter and to deduce from this a Limiting Absorption Principle for our operator (see \cite{FoSk}).  Here, we will use a third method which is, contrary to the others, a general method: the method of the weakly conjugate operator. With this method, we do not have to assume that the first order commutator is strictly positive but only positive and injective (see \cite{MaRi, Ri, BoGo}). 

Let $k\in\N^*$, $k\leq n$, and consider the decomposition $\R^n=\R^k\times \R^{n-k}$. With this decomposition, denote $(x,y)\in\R^n$ where $x\in\R^k$ and $y\in\R^{n-k}$. For $h\in C^\infty(\R^n)$, denote 
\[\nabla_x h(x,y)=\biggl(\partial_i h(x,y)\biggr)_{i=1,\cdots,k}.\]
 In \cite{MaRi}, using the method of weakly conjugate operator, M. M{\u{a}}ntoiu and S. Richard proved the following
\begin{theorem}[\cite{MaRi}, Theorem II.2]\label{th: MaRi}
Let $k\geq 2$. Let $H=\Delta+V$ whith $V\in C^\infty(\R^n)\cap L^\infty(\R^n)$ a real potential. Assume that
\begin{enumerate}
\item $x\cdot\nabla_x V\in L^\infty(\R^n)$.

\smallskip

\item For all $(x,y)\in\R^n$, $-x\cdot \nabla_x V(x,y)\geq 0$.

\smallskip

\item There is a constant $c$ such that, for all $(x,y)\in\R^n$,
\[\left| (x\cdot\nabla_x)^2 V (x,y)\right|\leq -c x\cdot \nabla_x V(x,y).\]
\end{enumerate}
Then, there exists a Banach space $\ca\subset L^2(\R^n)$ such that $\|(H-\lambda\pm \rmi \nu)^{-1}\|_{B(\ca,\ca^*)}$ is bounded uniformly in $\lambda\in\R$, $\nu>0$, where $\ca^*$ is the dual space of $\ca$.
\end{theorem}

These conditions do not permit to cover some situations: in fact, assumption (2) does not allow to have an oscillating potential of the form $V(x)=\sin(|x|^2)e^{-|x|^2}$. Moreover, because of the derivative, for an oscillating potential, $(x\cdot\nabla_x)^2 V$ can be unbounded. In this article, we will use the abstract result of the method of the weakly conjugate operator with different type of conjugate operators.

A first conjugate operator we use is the operator $A_F$ defined by 
\[A_F=\frac{1}{2}(p\cdot F(q)+F(q)\cdot p)\]
with $F$ a $C^\infty$ vector field with some good properties. Remark that this type of conjugate operator was already used by R. Lavine in \cite{La1,La2,La3,ABG}. This conjugate operator permits to apply the method of the weakly conjugate operator to potentials for which the derivative does not have enough decay at infinity. With this conjugate operator, we can prove the following:
\begin{theorem}\label{th: A_F intro}
Let $n\geq3$ and $0\leq\mu<(1+\frac{n}{n-2})^{-2}$.
Let $V_1,V_2\in L^1_{loc}(\R^n,\R)$ and $H=\Delta+V_1+\rmi V_2$. Assume that
\begin{enumerate}
\item $V_k$ are $\Delta$-compact and $V_2\geq0$;

\item $q\jap{q}^{-\mu}\cdot\nabla V_k$ are $\Delta$-compact;

\item There is \[C>-\frac{(n-2)^2(1-\mu(1+\frac{n}{n-2})^{2})}{2}\] such that $-x\cdot\nabla V_1(x)\geq \frac{C}{|x|^2}$ for all $x\in\R^n$;

\item There is $C'>0$ such that for all $x\in\R^n$, $|(x\jap{x}^{-\mu}\cdot\nabla)^2 V_k(x)|\leq C'|x|^{-2}\jap{x}^{-\mu}$.
\end{enumerate}
Then
\[\sup_{\lambda\in\R,\eta>0}\|\jap{q}^{-\mu/2}|q|^{-1}(H-\lambda+\rmi \eta)^{-1}|q|^{-1}\jap{q}^{-\mu/2}\|<\infty.\]

Moreover, $H$ does not have eigenvalue in $\R$.
\end{theorem}
We make few remarks about this theorem:
\begin{enumerate}
\item If $0<\mu$, we do not require to have $q\cdot\nabla V_i$ bounded.

\item Let $V_1(x)=\jap{q}^{-\alpha}$, $\alpha>0$ and $V_2(x)=\jap{x}^{-\beta}$, $\beta>0$. If we want to use the generator of dilations with these potentials, we can see that we have to assume that $\alpha,\beta\geq2$ to use the method of the weakly conjugate operator (see \cite[Theorem C.1]{BoGo}). Here, we only have to assume that there is $0\leq\mu<(1+\frac{n}{n-2})^{-2}$ such that $\alpha+\mu\geq2$ and $\beta+\mu\geq2$. In particular, if $\alpha>2-(1+\frac{n}{n-2})^{-2}$ and $\beta>2-(1+\frac{n}{n-2})^{-2}$, then Theorem \ref{th: A_F intro} applies.
\end{enumerate}

Another conjugate operator we use is the operator $A_u$ defined by
\[A_u=\frac{1}{2}(q\cdot u(p)+u(p)\cdot q)\]
where $u$ is a $C^\infty$ vector field with some good properties. Remark that this conjugate operator was already used in \cite{ABG}. Moreover, it turns out that this type of conjugate operator is particulary usefull when the potential has high oscillations because conditions on commutators does not impose derivatives (see \cite{Ma2,Ma1}). Using this conjugate operator, we obtain the following:
\begin{theorem}\label{th:sr cond}
Let $n\geq3$.
Let $V_1,V_2\in L^1_{loc}(\R^n,\R)$ and $H=\Delta+V_1+\rmi V_2$ with $V_2\geq0$.
Assume that $|q|^2V_1$ and $|q|V_1$ are bounded with bound small enough and that $\jap{q}^3V_i$ is bounded.
Then Theorem \ref{th: Theorem B.1} applies. In particular, for all $1\leq\mu<2$,
\[\sup_{\rho\in\R,\eta>0}\|\jap{p}^{-\mu/2}|q|^{-1}(H-\rho+\rmi \eta)^{-1}|q|^{-1}\jap{p}^{-\mu/2}\|<\infty.\]

Moreover, $H$ does not have eigenvalue in $\R$.
\end{theorem}

We will make few remarks about Theorem \ref{th:sr cond}
\begin{enumerate}
\item Let $W_1,W_2$ potentials such that $\jap{q}^3W_i$ is bounded and $W_2\geq0$. Let $w\in\R$. Then $V_1=w W_1$ and $V_2=W_2$ satisfy assumptions of Theorem \ref{th:sr cond} for $|w|$ small enough. 

\smallskip

\item In the case $V_2=0$ and $V_1\in L^{n/2}$, the absence of negative eigenvalues can be proved with an other method: using the Lieb-Thirring inequality (see \cite{Li}), we already know that if $V_1$ is small enough, the number of negative eigenvalue have to be $0$.

\smallskip

\item Using Sobolev inequalities and taking $u(x)=x\jap{x}^{-\mu}$, $1<\mu<2$, we can replace the assumption $\jap{q}^{3}V$ bounded by $\jap{q}^2V$ bounded and $x\mapsto|x|^3V(x)\in L^p$ with $ p\geq\frac{n}{\mu-1}$
(see \cite[Corollary 5.9]{Ma1}). 

\smallskip

\item Since assumptions on the potential does not impose conditions on the derivatives of the potential, we can use this result with potentials which have high oscillations. For example, if $V_2=0$ and $V_1(x)=w(1-\kappa(|x|))\frac{\sin(k|x|^\alpha)}{| x|^\beta}$ with $w,k,\alpha\in\R$, $\beta>0$ and $\kappa\in C^\infty_c(\R,\R)$ such that $\kappa=1$ on $[-1,1]$ and $0\leq \kappa\leq 1$, it suffices to suppose that $w$ is small enough and that $\beta\geq3$ to obtain a Limiting Absorption Principle on all $\R$. Remark that because of the oscillations, Theorem C.1 of \cite{BoGo} does not apply nor Theorem 1.1 of \cite{FoSk}.

\smallskip

\item Notice that the absence of eigenvalue was also proved for this type of potential. For example, we can see in \cite[Theorem 1 and Theorem 2]{FKV} that is is sufficient in dimension $n\geq3$ to assume that $|x|^2 V$ is bounded with bound small enough to prove the absence of eigenvalue. Here, we suppose more decay on the potential $V$ and the absence of real eigenvalue is only a consequence of the obtention of a Limiting Absorption Principle on all the real axis.
\end{enumerate}

The paper is organized as follows. In Section \ref{s:notation}, we will give some notations we will use below and we recall some basic fact about regularity with respect to a conjugate operator. In Section \ref{s:theoric}, we will recall the abstract result corresponding to Theorem \ref{th: MaRi}. In Section \ref{s:A_D}, we will recall a result concerning the application of the method of the weakly conjugate operator with the generator of dilations as conjugate operator, and we will see that, with this conjugate operator, we can avoid conditions on the second order derivatives. In Sections \ref{s:A_F} and \ref{s:A_u}, we will use the method of the weakly conjugate operator with $A_F$ and $A_u$ as conjugate operator. In Section \ref{s:dim 1}, we will see that we can use a conjugate operator which is a differential operator only in certain directions. In Section \ref{s:oscillant}, we will give examples of potentials for which our previous results apply. In Appendix \ref{s: H-S formula}, we recall the Helffer-Sjostrand formula and some properties of this formula we will use in the text.

\section{Notation and basic notions}\label{s:notation}

\subsection{Notation}
Let $X=\R^n$ and for $s\in\R$ let $\ch^s$ be the usual Sobolev
spaces on $X$ with $\ch^0=\ch=L^2(X)$ whose norm is denoted
$\|\cdot\|$. We are mainly interested in the space $\ch^{1}$ defined
by the norm $\|f\|_1^2=\int\left(|f(x)|^2+|\nabla f(x)|^2\right)\d x$ and
its dual space $\ch^{-1}$.  



We denote $q_j$ the operator of multiplication by the coordinate $x_j$, $p_j=-\rmi \partial_j$ and we denote $p=(p_j)_{j=1,\cdots,n}$ and $q=(q_j)_{j=1,\cdots,n}$ considered as operators in $\ch$. For $k\in X$ we denote
$k\cdot q=k_1q_1+\dots+k_\nu q_n$. If $u$ is a
measurable function on $X$ let $u(q)$ be the operator of
multiplication by $u$ in $\ch$ and $u(p)=\cf^{-1}u(q)\cf$, where $\cf$
is the Fourier transformation:
\[
(\cf f)(\xi)=(2\pi)^{-\frac{\nu}{2}} \int \rme^{-\rmi x\cdot\xi} u(x) \d x .
\]
If there is no ambiguity we keep the same notation for these operators
when considered as acting in other spaces.  

We are mainly interested in potentials $V$ which are multiplication operators in the following general sense.

\begin{definition}\label{df:mult} 
A map $V\in\rb$ is called a \emph{multiplication operator} if $V\rme^{\i k\cdot q}=\rme^{\i k\cdot q}V$ for all $k\in X$. Or, equivalently, if $V\theta(q)=\theta(q)V$ for all $\theta\in C_\rmc^\infty(X)$.
\end{definition}




As usual $\jap{x}=\sqrt{1+|x|^2}$. Then $\jap{q}$ is the operator of
multiplication by the function $x\mapsto\jap{x}$ and
$\jap{p}=\cf^{-1}\jap{q}\cf$.  For real $s,t$ we denote $\ch^t_s$ the
space defined by the norm
\begin{equation}\label{eq:K}
\|f\|_{\ch^t_s}= \|\jap{q}^s f\|_{\ch^t}= \|\jap{p}^t \jap{q}^s f\|=\| \jap{q}^s\jap{p}^t f\| .
\end{equation}
Note that the dual space of $\ch^t_s$ may be identified with
$\ch^{-t}_{-s}$.

\subsection{Regularity}

Let $F', F''$ be two Banach space and $T:F'\rightarrow F''$ a bounded operator.

Let $A$ a self-adjoint operator.

Let $k\in\N$. we say that $T\in C^k(A,F',F'')$ if, for all $f\in F'$, the map $\R\ni t\rightarrow e^{itA}Te^{-itA}f$ has the usual $C^k$ regularity. The following characterisation is available:

\begin{proposition}
 $T\in C^1(A,F',F'')$ if and only if $[T,A]$ has an extension in $\cb(F',F'')$.
\end{proposition}
It follows that, for $k>1$, $T\in C^k(A,F',F'')$ if and only if $[T,A]\in C^{k-1}(A,F',F'')$.


If $T$ is not bounded, we say that $T\in C^k(A,F',F'')$ if for one $z\notin\sigma(T)$, and thus for any $z\notin\sigma(T)$, $(T-z)^{-1}\in C^k(A,F',F'')$.

\begin{proposition}
For all $k>1$, we have
\[ C^{k}(A,F',F'')\subset C^{1,1}(A,F',F'')\subset C^{1}(A,F',F'').\]
\end{proposition}
If $F'=F''=\ch$ is an Hilbert space, we note $C^1(A)=C^1(A,\ch,\ch^*)$. If $T$ is not bounded, $T$ is of class $C^1(A)$ if and only if $[T,\rmi A]:\cd(T)\rightarrow\cd(T)^*$ is bounded and, for some $z\in\C\backslash\sigma(T)$, the set $\{f\in\cd(A), R(z)f\in\cd(A)\text{ and }R(\bar{z})f\in\cd(A)\}$ is a core for $A$. Remark that, in general, because of the second assumption, it is more difficult to show that $T$ is of class $C^1(A)$ than to show that $T$ is of class $C^1(A,\cd(T),\cd(T)^*)$. This is not the case if we suppose that the unitary group generated by $A$ leaves $\cd(T)$ invariant.
For $T$ is self-adjoint, we have the following:

\begin{theorem}[Theorem 6.3.4 from \cite{ABG}]\label{th:634}
Let $A$ and $T$ be self-adjoint operator in a Hilbert space $\ch$. Assume that the unitary group $\{\exp(\rmi A\tau)\}_{\tau\in\R}$ leaves the domain $D(T)$ of $T$ invariant. Set $\cg=D(T)$ endowed with it graph topology. Then 
\begin{enumerate}
\item $T$ is of class $C^1(A)$ if and only if $T\in C^1(A,\cg,\cg^*)$.

\item $T$ is of class $C^{1,1}(A)$ if and only if $T\in C^{1,1}(A,\cg,\cg^*)$.
\end{enumerate}
\end{theorem}

If $\cg$ is the form domain of $H$, we have the following:
\begin{proposition}[see p. 258 of \cite{ABG}]
Let $A$ and $T$ be self-adjoint operators in a Hilbert space $\ch$. Assume that the unitary group $\{\exp(\rmi A\tau)\}_{\tau\in\R}$ leaves the form domain $\cg$ of $T$ invariant. Then $T$ is of class $C^k(A)$ if and only if $T\in C^k(A,\cg,\cg^*)$, for all $k\in\N$.
\end{proposition}

As previously, since $T:\cg\rightarrow\cg^*$ is always bounded, it is, in general, easier to prove that $T\in C^k(A,\cg,\cg^*)$ than $T\in C^k(A)$.

\subsection{The Hardy inequality}

To have concrete conditions on the potential, we will use the Hardy inequality. For this reason we recall it:
\begin{proposition}
Assume that $n\geq3$. Let $f\in \ch^1(\R^n)$. We have
\[\frac{(n-2)^2}{4}\int_{\R^n}\frac{1}{|x|^2}|f(x)|^2dx\leq \int_{\R^n}|\nabla f(x)|^2dx.\]
\end{proposition}
In particular, this inequality implies that if $B(q)$ is a multiplication operator such that $|q|^2B(q)$ is bounded, if $n\geq3$, then, there is $C>0$ such that
\[\left|(f,B(q)f)\right|\leq C\|\nabla f\|^2.\]

 \section{The method of the weakly conjugate operator}\label{s:theoric}
 
 In this section, we will recall a version of the Mourre theory in order to obtain a limiting absorption principle near thresholds, called the method of the weakly conjugate operator. This Mourre theory was developed by A. Boutet de Monvel and M. Mantoiu in \cite{BM}. An improvement of this theory was developped by S. Richard in \cite{Ri} fo the self-adjoint case. Here, we recall a version of this theory present in \cite{BoGo} adapted to the non-self-adjoint case.
 
 Let $H^\pm$ two closed operators with a common domain $\cd$. We suppose that $(H^+)^*=H^-$. Since $H^\pm$ are densely defined, has common domain and are adjoint of the other, $\Re(H^\pm)$ and $\Im(H^\pm)$ are closable and symmetric on $\cd$. Even if they are not self-adjoint, we can remark that $\cd$ is a core for them. Therefore $\cg$ is a core for them too. We keep the same notation for their closure.
 
 We assume that $H^+$ is dissipative i.e $\Im(H^+)\geq0$. This implies that $\Im(H^-)\leq0$ and, by the numerical range theorem, we can say that $\sigma(H^\pm)$ is include in the half-plane $\{z\in\C,\pm\Im(z)\geq0\}$. Let $S$ a non negative, injective, self-adjoint operator with form domain $\cg=\cd(S^{1/2})\supset\cd$. Let $\cs$ the completion of $\cg$ under the norm $\|f\|_\cs=\left(\|S^{1/2}f\|\right)^{1/2}$. 
 
 We get the following inclusions with continuous and dense embeddings
 \[\cd\subset\cg\subset\cs\subset\cs^*\subset\cg^*\subset\cd^*.\]
 We will need an external operator $A$, the conjugate operator. Assume $A$ is self-adjoint in $\ch$ and $S\in C^1(A)$. Let $W_t=e^{\rmi tA}$ be the $C_0$-group associated to A in $\ch$. We suppose that $W_t$ stabilizes $\cg$ and $\cs$. This implies, by duality, that $W_t$ stabilizes $\cg^*$ and $\cs^*$. Remark that if $[S,\rmi A]:\cs\rightarrow\cs^*$ is bounded, then the invariance of $\cs$ under the $C^0$-group $W_t$ is a consequence of the invariance of $\cg$ (see \cite[Remark B.1]{BoGo}).
 
 \begin{theorem}[Theorem B.1 of \cite{BoGo}]\label{th: Theorem B.1}
 Let $H^\pm$ and $A$ as above. Suppose that $H^\pm\in C^2(A,\cg,\cg^*)$ and that there is $c>0$ such that
 \begin{equation}\label{eq: B.7}
 \left|(H^\mp f,Ag)-(Af,H^\pm g)\right|\leq c\|f\|\cdot\|(H^\pm\pm\rmi)g\|,
 \end{equation}
 for all $f,g\in\cd\cap\cd(A)$.
 Assume that $\pm\Im(H^\pm)\geq0$ and that there is $c_1\geq0$ such that
 \begin{equation}\label{eq: B.8}
 [\Re(H^\pm),\rmi A]-c_1\Re(H^\pm)\geq S>0,
 \end{equation}
 \[\pm c_1[\Im(H^\pm),iA]\geq0,\]
 in sense of forms on $\cg$. Suppose also there exists $C>0$ such that
 \begin{equation}\label{eq: B.10}
 \left|(f,[[H^\pm,A],A]f)\right|\leq C\|S^{1/2}f\|^2
 \end{equation}
 for all $f\in\cg$.
 Then, there are $c'$ and $\mu_0>0$ such that
 \begin{equation}\label{eq: B.11}
 \left|(f,(H^\pm-\lambda\pm\rmi \mu)^{-1}f)\right|\leq c'\left(\|S^{-1/2}f\|^2+\|S^{-1/2}Af\|^2\right)
 \end{equation}
 for all $0<\mu<\mu_0$ and $\lambda\geq0$ if $c_1>0$, and $\lambda\in\R$ if $c_1=0$.
\end{theorem} 
 Remark that, if we want to have \eqref{eq: B.8} with $c_1>0$, it seems necessary that $[H,\rmi A]$ reproduce $H$. This is the case with the generator of dilations for which we have $[\Delta,\rmi A_D]=2\Delta$. But with conjugate operators we want to use, the commutator with the Laplacian does not reproduce the Laplacian $\Delta$. For this reason, we will use Theorem \ref{th: Theorem B.1} only in the case $c_1=0$.

\section{The generator of dilations as conjugate operator}\label{s:A_D}
In this section, we will see what conditions are sufficient to apply Theorem \ref{th: Theorem B.1} with the generator of dilations as conjugate operator. For this, we will recall a result from \cite{BoGo} in dimension higher or equal to $3$ which illustrate the method of the weakly conjugate operator and we will give a variation of this result.  We will also recall a result from B. Simon which show that dimensions 1 and 2 are quite particular.

Using the Hardy inequality and the generator of dilations as conjugate operator, one can show the following result
 
 \begin{theorem}[Theorem C.1, \cite{BoGo}]\label{th:BoGo}
 Let $n\geq3$. 
 Assume that $V_1,V_2\in L^1_{loc}(\R^n,\R)$ satisfy:
 \begin{enumerate}
 \item $V_k$ are $\Delta$-bounded with bound less than one and $V_2\geq0$;
 
 \smallskip
 
 \item $\nabla V_k, q\cdot\nabla V_k$ are $\Delta$-bounded and $|q|^2(q\cdot \nabla)^2V_k$ are bounded;
 
 \smallskip
 
 \item There is $c_1\in[0,2)$ and $C\in[0,\frac{(2-c_1)(n-2)^2}{4})$ such that
 \[x\cdot\nabla V_1(x)+c_1 V_1(x)\leq \frac{C}{|x|^2}\]
 and 
 \[-c_1 x\cdot\nabla V_2(x)\geq 0\]
 for all $x\in\R^n$.
 \end{enumerate}
 Then $H$ has no eigenvalue in $[0,+\infty)$ and 
 \begin{equation}\label{eq:C.1}
 \sup_{\lambda\in[0,+\infty),\mu>0}\||q|^{-1}(H-\lambda+\rmi \mu)^{-1}|q|^{-1}\|<\infty.
 \end{equation}
 If $c_1=0$, $H$ has no eigenvalue in $\R$ and \eqref{eq:C.1} holds true for $\lambda\in\R$.
 \end{theorem}
 
 Using that $A_D$ can be writen as $A_D=q\cdot p-\frac{n\rmi}{2}=p\cdot q+\frac{n\rmi}{2}$, we can avoid the condition on the second order derivative of $V_1$ and $V_2$ to obtain the following
 \begin{theorem}\label{th:A_F 2}
Let $n\geq3$.
Let $V_1,V_2\in L^1_{loc}(\R^n,\R)$ and $H=\Delta+V_1+\rmi V_2$.
Suppose that
\begin{enumerate}[label=(\textbf{H.\arabic*}),ref=\textbf{(H.\arabic*)}]
\item $V_k$ are $\Delta$-bounded with bound less than $1$ and $V_2\geq0$;

\item $\nabla V_k$ and $q\nabla V_k$ are $\Delta$-bounded and 
\[|x\nabla V_k(x)|\leq \frac{C}{|x|^2}\quad \forall x\in\R^n\backslash\{0\};\]

\item There is $c_1\in[0,2)$ and $C\in[0,\frac{(2-c_1)(n-2)^2}{4})$ such that
 \[x\cdot\nabla V_1(x)+c_1 V_1(x)\leq \frac{C}{|x|^2}\]
  and 
 \[-c_1 x\cdot\nabla V_2(x)\geq 0\]
 for all $x\in\R^n$. 
 \end{enumerate}
 Then $H$ has no eigenvalue in $[0,+\infty)$ and 
 \begin{equation}\label{eq:C.2}
 \sup_{\lambda\in[0,+\infty),\mu>0}\||q|^{-1}(H-\lambda+\rmi \mu)^{-1}|q|^{-1}\|<\infty.
 \end{equation}
 If $c_1=0$, $H$ has no eigenvalue in $\R$ and \eqref{eq:C.2} holds true for $\lambda\in\R$.
\end{theorem}

We make few remarks about Theorem \ref{th:A_F 2}:
\begin{enumerate}
\item We do not assume any conditions on the second order derivatives which can be usefull if $V$ is a multiplication by a function wich is not $C^2$ or if $V$ is an oscillating potential (see Theorem \ref{th:oscillant}). 

\smallskip

\item Remark that if $|q|^2q\cdot \nabla V_1$ is bounded with bound small enough, then Assumption (H.2) and (H.3) are satisfied with $c_1=0$.

\smallskip

\item Remark that if $V$ is a short range type potential, $q\nabla V_k$ is not necessary $\Delta$-bounded. For this reason,  Theorems \ref{th:BoGo} and \ref{th:A_F 2} do not apply to short range potentials.
\end{enumerate}

\begin{proof}[Theorem \ref{th:A_F 2}]
Since the proof is quite similar to the proof of  Theorem \ref{th:BoGo}, we will only explain what changes. In particular, for conditions on the first order commutator (Assumptions (H.3)), nothing changes. The idea is to prove, using Hardy inequality that the first order commutator is positive and that $S$ has the form $c\Delta$. In \cite{BoGo}, they use, in a second time, that $V$ is of class $C^2(A_D)$ and the Hardy inequality to show the second order commutator estimate.  Since we want to use $S=c\Delta$ with $c=2-c_1$, we can see that $\cg=\ch^1$ and $\|f\|_\cs=\|pf\|_{L^2}$. Remark that $e^{\rmi tA_D}\ch^1\subset\ch^1$. In particular, we do not need to suppose that $[[V,\rmi A_D],\rmi A_D]$ is $\Delta$-bounded to obtain the regularity $C^2(A_D)$ but only that this commutator is bounded on $\cg$ to $\cg^*$.

For $V=V_1,V_2$, we have
\[[V,\rmi A_D]=-q\cdot\nabla V.\]
Thus 
\begin{eqnarray}\label{eq:commut A_D}
[[V,\rmi A_D],\rmi A_D]&=&-[q\cdot\nabla V,\rmi A_D]\nonumber\\
&=&-\rmi q\cdot\nabla VA_D+\rmi A_Dq\cdot\nabla V\nonumber\\
&=&-\rmi (q\cdot\nabla V)(q\cdot p)+\rmi (p\cdot q)(q\cdot\nabla V)-nq\cdot\nabla V.
\end{eqnarray}
In particular, if we assume that $|q|q\cdot\nabla V_k$ is bounded, then $H\in C^2(A_D,\ch^1,\ch^{-1})$.

Let $f\in\cs$. For the first term of the right side of \eqref{eq:commut A_D}, we have:
\begin{eqnarray*}
\left|(f,(q\cdot\nabla V)(q\cdot p)f)\right|&=&\left|((q\cdot\nabla V)qf,pf)\right|\\
&\leq&\|(q\cdot\nabla V)qf\|\|pf\|\\
&\leq&\||q|^2(q\cdot\nabla V)\|_\infty\||q|^{-1}f\|\|pf\|\\
&\leq&\frac{2}{n-2)}\||q|^2(q\cdot\nabla V)\|_\infty\|pf\|^2.
\end{eqnarray*}
A similar proof can be made for the second term. For the last term, from Hardy inequality, we deduce that 
\[\left|(f,(q\cdot\nabla V)f)\right|\leq\frac{4}{(n-2)^2}\||q|^2(q\cdot\nabla V)\|_\infty\|pf\|^2.\]
This implies that the second order comutator is bounded from $\cs$ to $\cs^*$. For the rest of the proof, we follow the proof of Theorem C.1 of \cite{BoGo}.
\qed
\end{proof}
 
 Remark that in all cases, we assume that the dimension $n$ is higher than $3$ to use the Hardy inequality. The case of dimension $n=1$ or $n=2$ is quite particular. In fact, in dimension $n=1$ or $n=2$, for a large class of potential, we can prove the existence of a negative eigenvalue which is a contradiction with the result which said that if $c_1=0$, $H$ has no real eigenvalue.
 \begin{theorem}[Theorem 2.5,\cite{Si3}]
 Let $n=1$. Let $V$ obey $\int(1+x^2)|V(x)|dx<\infty$, $V$ not a.e zero. Then $H=\Delta+\lambda V$ has a negative eigenvalue for all $\lambda>0$ if and only if $\int V(x)dx\leq0$.
 \end{theorem}
 Similarly in dimension 2, we have
 \begin{theorem}[Theorem 3.4, \cite{Si3}]
 Let $n=2$. Let $V$ obey $\int|V(x)|^{1+\delta}d^2x<\infty$ and $\int(1+|x|^\delta)|V(x)|d^2x<\infty$ for some $\delta>0$, $V$ not a.e zero. Then $H=\Delta+\lambda V$ has a negative eigenvalue for all small positive $\lambda$ if and only if $\int V(x)dx\leq0$.
 \end{theorem}
 These dimensions are very different from the others. In fact, by the Lieb-Thirring inequality (see \cite{Li}), we know that, in dimension $n\geq3$, if the negative part of $V$ is in $L^{n/2}(\R^n)$ with norm small enough, then $H=\Delta+V$ does not have any negative eigenvalue. This two results imply that in dimensions 1 and 2, a Schr\"odinger operator with a non positive potential which satisfies assumptions like in Theorems \ref{th:BoGo} and \ref{th:A_F 2} can have negative eigenvalue. Thus, we have to assume some positivity of the first order commutator or to use $c_1>0$.

\section{A conjugate operator with decay in the position variable}\label{s:A_F}
Now we will see how we can change the conjugate operator to obtain other conditions on the potential which require less decay on the derivatives of potentials. To do this, we will apply Theorem \ref{th: Theorem B.1}  with the following conjugate operator $A_F$.

Let $n\geq3$, $0\leq\mu<1$ and let $F(q)=q\jap{q}^{-\mu}$. Let
\[A_F=\frac{1}{2}(p\cdot F(q)+F(q)\cdot p)=\jap{q}^{-\mu/2}A_D\jap{q}^{-\mu/2}.\]
Notice that, by Proposition 4.2.3 of \cite{ABG}, we know that $A_F$ is essentially self-adjoint on $C^\infty_c$.
By a simple computation on the form domain $C_c^\infty$, we have
\begin{eqnarray*}
[\Delta,\rmi A_F]&=&[\Delta,\rmi \jap{q}^{-\mu/2}]A_D\jap{q}^{-\mu/2}+\jap{q}^{-\mu/2}[\Delta,\rmi A_D]\jap{q}^{-\mu/2}\\
& &+\jap{q}^{-\mu/2}A_D[\Delta,\rmi \jap{q}^{-\mu/2}]\\
&=&2\jap{q}^{-\mu/2}\left(\Delta-\mu A_D\jap{q}^{-2}A_D\right)\jap{q}^{-\mu/2}.
\end{eqnarray*}
For all $f\in\ch^1$, by Hardy inequality, we have
\begin{eqnarray*}
(f,A_D\jap{q}^{-2}A_Df)&=&\|\jap{q}^{-1}A_Df\|^2\\
&\leq&\left(\|q\jap{q}^{-1}\cdot pf\|+\frac{n}{2}\|\jap{q}^{-1}f\|\right)^2\\
&\leq&\left(1+\frac{n}{n-2}\right)^2\|\nabla f\|^2.
\end{eqnarray*}
In particular, if $0\leq\mu<(1+\frac{n}{n-2})^{-2}$, 
\[[\Delta,\rmi A_F]\geq 2(1-\mu(1+\frac{n}{n-2})^{2})\jap{q}^{-\mu/2}\Delta\jap{q}^{-\mu/2}\geq 0.\]
Thus, we can take $S=c\jap{q}^{-\mu/2}\Delta\jap{q}^{-\mu/2}$ with $c>0$ with domain $\cd(S)=\ch^2$. Remark that since $S^{1/2}=\sqrt{c} |p| \jap{q}^{-\mu/2}$, $\cg=\cd(S^{1/2})=\ch^1$. In particular, by Proposition 4.2.4 of \cite{ABG}, we know that the $C_0$-group associated to $A_F$ leaves $\cg$ invariant. To prove that the $C_0$-group associated to $A_F$ leaves $\cs$ invariant, we will show that $[S,\rmi A_F]$ is bounded from $\cs$ to $\cs^*$. We will consider this commutator as a form with domain $C^\infty_c$ and we use the same notation for its closure. By a simple computation, 
\begin{eqnarray*}
[S,\rmi A_F]&=&c[\jap{q}^{-\mu/2},\rmi A_F]\Delta\jap{q}^{-\mu/2}+c\jap{q}^{-\mu/2}[\Delta,\rmi A_F]\jap{q}^{-\mu/2}\\
& &+c\jap{q}^{-\mu/2}\Delta[\jap{q}^{-\mu/2},\rmi A_F]\\
&=&c\jap{q}^{-\mu/2}[\Delta,\rmi A_F]\jap{q}^{-\mu/2}\\
& &+c\frac{\mu}{2}\left(\jap{q}^{-\mu/2}|q|^2\jap{q}^{-2}\Delta\jap{q}^{-\mu/2}+\jap{q}^{-\mu/2}\Delta|q|^2\jap{q}^{-2}\jap{q}^{-\mu/2}\right).
\end{eqnarray*}  
Using the form of $[\Delta,\rmi A_F]$, by a simple computation, we can see that the first term on the right hand side is bounded from $\cs$ to $\cs^*$. For the other term, by a simple computation, we have:
\begin{eqnarray*}
& &\jap{q}^{-\mu/2}|q|^2\jap{q}^{-2}\Delta\jap{q}^{-\mu/2}+\jap{q}^{-\mu/2}\Delta|q|^2\jap{q}^{-2}\jap{q}^{-\mu/2}\\
&=&\jap{q}^{-\mu/2}\left(2p|q|^2\jap{q}^{-2}p+[p,[p,|q|^2\jap{q}^{-2}]\right)\jap{q}^{-\mu/2}
\end{eqnarray*}

By Hardy inequality, since $|q|^2 [p,[p,|q|^2\jap{q}^{-2}]$ is bounded, we can see that this term is bounded from $\cs$ to $\cs^*$. Thus, by \cite[Remark B.1]{BoGo}, the $C_0$-group associated to $A_F$ leaves $\cs$ invariant. Therefore, we have the following:
\begin{theorem}\label{thm: AF3}
Let $0\leq\mu<(1+\frac{n}{n-2})^{-2}$.
Let $V_1,V_2\in L^1_{loc}(\R^n,\R)$ and $H=\Delta+V_1+\rmi V_2$. Assume that
\begin{enumerate}
\item $V_k$ are $\Delta$-compact and $V_2\geq0$;

\item $q\jap{q}^{-\mu}\cdot\nabla V_k$ are $\Delta$-compact;

\item There is \[C>-\frac{(n-2)^2(1-\mu(1+\frac{n}{n-2})^{2})}{2}\] such that $-x\cdot\nabla V_1(x)\geq \frac{C}{|x|^2}$ for all $x\in\R^n$;

\item There is $C'>0$ such that for all $x\in\R^n$, $|(x\jap{x}^{-\mu}\cdot\nabla)^2 V_k(x)|\leq C'|x|^{-2}\jap{x}^{-\mu}$.
\end{enumerate}
Then Theorem \ref{th: Theorem B.1} applies and
\[\sup_{\lambda\in\R,\eta>0}\|\jap{q}^{-\mu/2}|q|^{-1}(H-\lambda+\rmi \eta)^{-1}|q|^{-1}\jap{q}^{-\mu/2}\|<\infty.\]

Moreover, $H$ does not have eigenvalue in $\R$.
\end{theorem}
To prove this theorem, we use the abstract result of the method of the weakly conjugate operator and the Hardy inequality as in the proof of Theorems \ref{th:BoGo} and \ref{th:A_F 2}.

If $n=1$, as we saw in the previous section, it is not sufficient to suppose only that the derivatives of $V$ have sufficient decay at infinity. To avoid the possible negative eigenvalue, we can assume some positivity of the first order commutator of the potential $V_1$. To simplify notation, remark that 
\begin{eqnarray*}
F'(x)&=&(1-\mu)\jap{x}^{-\mu}+\mu\jap{x}^{-\mu-2},\\
F''(x)&=&-\mu x\jap{x}^{-\mu-2}\left(1-\mu+(\mu+2)\jap{x}^{-2}\right),\\
F'''(x)&=&\mu(1-\mu)(1+\mu)\jap{x}^{-\mu-2}+4\mu(\mu+2)\jap{x}^{-\mu-4}\\
& &+\mu(\mu+2)(\mu+4)\jap{x}^{-\mu-6} 
\end{eqnarray*}
\begin{theorem}\label{th: A_F dim1}
Let  $0\leq\mu\leq1$. Let $F(x)=x\jap{x}^{-\mu}$.
Let $V_1,V_2\in L^1_{loc}(\R,\R)$ and $H=\Delta+V_1+\rmi V_2$. Assume that
\begin{enumerate}
\item $V_i$ are $\Delta$-compact and $V_2\geq0$;

\item $F(q) V_i'$ are $\Delta$-compact;

\item \[W(x)=-F(x) V_1'(x)-\frac{1}{2}F'''(x)\geq0\] for all $x\in\R$;

\item There is $C_1,C_2>0$ such that
\[\biggl|2F(x)W'(x)+F'''(x)F'(x)+(F''(x))^2\biggr|\leq C_1 W(x),\forall x\in\R\]
and 
\[\biggl|F(x)^2 V_2''(x)+F(x)F'(x)V_2'(x)\biggr|\leq C_2 W(x),\forall x\in\R.\]
\end{enumerate}
Then Theorem \ref{th: Theorem B.1} applies and there are $c>0$ and $\mu_0>0$ such that
\[\left|(f,(H-\lambda+\rmi \eta)^{-1}f)\right|\leq c\left(\|S^{-1/2}f\|^2+\|S^{-1/2}A_Ff\|^2\right),\]
with $S=2pF'(q)p+W(q)$ and $A_F=\frac{1}{2}(pF(q)+F(q) p)$.

Moreover, $H$ does not have eigenvalue in $\R$.
\end{theorem}

Remark that, a priori, if $\mu>0$, we do not impose that $xV_i'(x)$ is bounded; Theorem \ref{th: A_F dim1} applies if $x V_i'(x)$ as the same size as $\jap{x}^\mu$. In particular, if $\mu=1$, we only require that $V_i'$ is a bounded function.

\begin{proof}[Theorem \ref{th: A_F dim1}]
To prove this result, we only have to show that our assumptions imply assumptions of Theorem \ref{th: Theorem B.1}. Remark that we can write 
\[A_F=\frac{1}{2}(pF(q)+F(q) p)=pF(q)+\frac{\rmi}{2}F'(q)=F(q)p-\frac{\rmi}{2}F'(q).\]
By a simple computation, we have:
\begin{eqnarray*}
[\Delta,\rmi A_F]&=&[p^2,\rmi A_F]\\
&=&p[p,\rmi A_F]+[p,\rmi A_F]p\\
&=&p[p,\rmi F(q)p]+[p,\rmi pF(q)]p+\frac{\rmi}{2}\left([p,\rmi F'(q)]p-p[p,\rmi F'(q)]\right)\\
&=&2p[p,\rmi F(q)]p-\frac{1}{2}[p,\rmi[p,\rmi F'(q)]]\\
&=&2pF'(q)p-\frac{1}{2} F'''(q).
\end{eqnarray*}
Thus, we have:
\begin{eqnarray*}
[H,\rmi A_F]&=&2pF'(q)p-\frac{1}{2} F'''(q)-F(q)V_1'(q)-\rmi F(q)V_2'(q)\\
&=&2pF'(q)p+W(q)-\rmi F(q)V_2'(q)\\
&=& S-\rmi F(q)V_2'(q).
\end{eqnarray*}
Therefore, by assumptions, we know that \eqref{eq: B.7} and \eqref{eq: B.8} are satisfied. To prove that \eqref{eq: B.10} is true, we have to calculate the second order commutator. Since $W(q)$ is a multiplication operator, we have
\[[W,\rmi A_F]=-F(q)W'(q).\]
By a similar calculus on $-\rmi F(q)V_2'(q)$, we deduce that, by assumptions, $[[V_2,\rmi A_F],\rmi A_F]$ is bounded from $\cs$ to $\cs^*$.

For the last part of the second order commutator, we have:
\begin{eqnarray*}
[pF'(q)p,\rmi A_F]&=&[p,\rmi A_F]F'(q)p+p[F'(q),\rmi A_F]p+pF'(q)[p,\rmi A_F]\\
&=&p[p,\rmi F(q)]F'(q)p+\frac{\rmi}{2}[p,\rmi F'(q)]F'(q)p-pF(q)F''(q)p\\
& &+pF'(q)[p,\rmi F(q)]p-\frac{\rmi}{2}pF'(q)[p,\rmi F'(q)]\\
&=&p\left(2F'(q)^2-F(q)F''(q)\right)p+\frac{\rmi}{2}\left(F''(q)F'(q)p-pF'(q)F''(q)\right)\\
&=&p\left(2F'(q)^2-F(q)F''(q)\right)p-\frac{1}{2}[p,\rmi F''(q)F'(q)]\\
&=&p\left(2F'(q)^2-F(q)F''(q)\right)p-\frac{1}{2}\left(F'''(q)F'(q)+F''(q)^2\right).
\end{eqnarray*}
Thus, we have
\begin{eqnarray*}
[[H,\rmi A_F],\rmi A_F]&=&p\left(2F'(q)^2-F(q)F''(q)\right)p\\
& &-\frac{1}{2}\left(2F(q)W'(q)+F'''(q)F'(q)+F''(q)^2\right)+\rmi[[V_2,\rmi A_F],\rmi A_F].
\end{eqnarray*}
Since $(2F'(q)^2-F(q)F''(q))F'(q)^{-1}$ is bounded, by assumptions, \eqref{eq: B.10} is satisfied and, thus, Theorem \ref{th: A_F dim1} is a consequence of Theorem \ref{th: Theorem B.1}
\qed
\end{proof}

\section{A conjugate operator with decay in the momentum variable}\label{s:A_u}
In this section, we will prove Theorem \ref{th:sr cond}.

Let $\lambda:\R^n\rightarrow\R$ a positive function of class $C^\infty$, bounded with all derivatives bounded. We assume moreover that, for all, $i=1,\cdots,n$, $x_i\partial_{x_i}\lambda(x)$ is bounded.

Let 
\[A_u=\frac{1}{2}(q\cdot p\lambda(p)+p\lambda(p)\cdot q).\]
By \cite[Proposition 7.6.3]{ABG}, we know that $A_u$ is essentially self-adjoint on $L^2(\R^n)$ with domain $C^\infty_c$. Moreover,
\[[\Delta,\rmi A_u]=2\Delta\lambda(p).\]
Thus, let $S=c\Delta\lambda(p)$ with $c\in (0,2]$. Since $\lambda(p)>0$, $\lambda(p)$ is injective. This implies that $S$ is injective and positive. Moreover $\cg=\cd(S^{1/2})=\ch^{m}$ with $m\in[0,1]$. In particular, since $\exp(\rmi tA_u)$ leaves $\ch^t_s$ invariants (see \cite[Proposition 4.2.4]{ABG}), it also leaves $\cg$ invariant. Moreover, if we denote $S_1=\Delta \lambda(p)$, we have:

\[
[S_1,\rmi A_u]=2\Delta \lambda^2(p)+\lambda(p)\sum_{k=1}^n p_k^3\partial_{x_k}\lambda(p).
\]
In particular, for $f\in\cd(S)\cap\cd(A_u)$:
\begin{eqnarray*}
\left|(f,[S_1,\rmi A_u]f)\right|&\leq&2\|S_1^{1/2}f\|\|\lambda(p)S_1^{1/2}f\|\\
& &+\sum_{k=1}^n\left|(p_k\lambda^{1/2}(p)f,p_k\partial_{x_k}\lambda(p)p_k\lambda^{1/2}(p)f)\right|\\
&\leq&2\|\lambda(p)\|\|S_1^{1/2}f\|^2+\sum_{k=1}^n\|p_k\partial_{x_k}\lambda(p)\|\|p_k\lambda^{1/2}(p)f\|^2\\
&\leq&\left(2\|\lambda(p)\|+\sup_{k}\|p_k\partial_{x_k}\lambda(p)\|\right)\|S_1^{1/2}f\|^2.
\end{eqnarray*} 
This implies $S\in C^1(A_u,\cs,\cs^*)$ which implies by \cite[Remark B.1]{BoGo} the invariance of $\cs$ under the unitary group generated by $A_u$. Therefore  $S\in C^1(A_u)$ and by \cite[Remark B.1]{BoGo}, we can show that $\exp(\rmi tA_u)$ leaves $\cs$ invariant.
Using $A_u$ as conjugate operator in Theorem \ref{th: Theorem B.1}, we have the following
\begin{lemma}\label{l:A_u}
Let $V_1,V_2\in L^1_{loc}(\R^n,\R)$ and $H=\Delta+V_1+\rmi V_2$.
Assume that
\begin{enumerate}[label=(\textbf{H.\arabic*}),ref=\textbf{(H.\arabic*)}]
\item $V_i$ are $\Delta$-bounded with bound less than $1$ and $V_2\geq0$;

\item $ [V_i,\rmi A_u]$ and $[[V,\rmi A_u],\rmi A_u]$ are $\Delta$-bounded (or $\ch^1\rightarrow\ch^{-1}$ bounded) and 
\begin{equation}\label{eq:A_u ordre 2}
\left|(f, [[V_i,\rmi A_u],\rmi A_u]f)\right|\leq C\|p\lambda^{1/2}(p)f\|^2;
\end{equation}

\item There is $C'<2$ such that
\begin{equation}\label{eq:Virial A_u}
(f,[V_1,\rmi A_u]f)\geq C'\|p\lambda^{1/2}(p)f\|^2.
\end{equation}
\end{enumerate}
Then Theorem \ref{th: Theorem B.1} applies and
\[\sup_{\rho\in\R,\eta>0}\|\lambda^{1/2}(p)|q|^{-1}(H-\rho+\rmi \eta)^{-1}|q|^{-1}\lambda^{1/2}(p)\|<\infty.\]

Moreover, $H$ does not have eigenvalue in $\R$.
\end{lemma}

As in the previous section, we will give some concrete conditions on the potential which permits to apply Theorem \ref{th: Theorem B.1}.
\begin{theorem}
Let $n\geq3$.
Let $V_1,V_2\in L^1_{loc}(\R^n,\R)$ and $H=\Delta+V_1+\rmi V_2$ with $V_2\geq0$.
Assume that $|q|^2V_1$ and $|q|V_1$ are bounded with bound small enough and that $\jap{q}^3V_i$ is bounded.
Then Theorem \ref{th: Theorem B.1} applies. In particular, for all $1\leq\mu<2$,
\[\sup_{\rho\in\R,\eta>0}\|\jap{p}^{-\mu/2}|q|^{-1}(H-\rho+\rmi \eta)^{-1}|q|^{-1}\jap{p}^{-\mu/2}\|<\infty.\]

Moreover, $H$ does not have eigenvalue in $\R$.
\end{theorem}
\begin{proof}
Let $0<\mu<2$ and $\lambda(p)=\jap{p}^{-\mu}$.
We can write:
\[A_u=u(p)\cdot q+\frac{i}{2}(div u)(p)=q\cdot u(p)-\frac{i}{2}(div u)(p).\]
To alleviate the notations, let $S_1=p\jap{p}^{-\mu/2}$. Assume that $|q|^3V_i$ is bounded. Remark that, by assumptions on $\lambda$, we ever prove that $\Delta$ is of class $C^2(A_u,\ch^2,L^2)$. If $\mu\geq 1$, we can show that if $V_i$ satisfies $|q|^2V_i$ is bounded, then $V_i$ is of class $C^2(A_u,\ch^2,L^2)$ (voir \cite{Ma1}). This implies by sum that $H=\Delta+V_1+\rmi V_2$ is of class $C^2(A_u,\ch^2,L^2)$.
In particular, for $V=V_i$,
\begin{eqnarray}\label{eq:commut A_u}
[V,\rmi A_u]&=&V \rmi A_u-\rmi A_u V\nonumber\\
&=&qV\cdot iu(p)+\frac{1}{2}V(div u)(p)-i u(p)\cdot qV+\frac{1}{2}(div u)(p)V\nonumber\\
&=&[qV,i u(p)]+\frac{1}{2}\left(V(div u)(p)+(div u)(p)V\right).
\end{eqnarray}
For the first term, we will use the Helffer-Sj\"ostrand formula (see \ref{s: H-S formula}):
\begin{eqnarray*}
[qV,iu(p)]&=&\frac{i}{2\pi}\int_\C\frac{\partial \phi^\C}{\partial \bar{z}}(z-p)^{-1}[qV,ip](z-p)^{-1}dz\wedge d\bar{z}\\
&=&-\frac{1}{2\pi}p\int_\C\frac{\partial \phi^\C}{\partial \bar{z}}(z-p)^{-1}qV(z-p)^{-1}dz\wedge d\bar{z}\\
& &+\frac{1}{2\pi}\int_\C\frac{\partial \phi^\C}{\partial \bar{z}}(z-p)^{-1}qV(z-p)^{-1}dz\wedge d\bar{z} p\\
&=&-\frac{1}{2\pi}S_1\jap{p}^{\mu/2}\int_\C\frac{\partial \phi^\C}{\partial \bar{z}}(z-p)^{-1}qV(z-p)^{-1}dz\wedge d\bar{z}\\
& &\jap{p}^{\mu/2}(p+i)^{-1}(S_1+\rmi \jap{p}^{-\mu/2})\\
& &+\frac{1}{2\pi}(S_1+\rmi \jap{p}^{-\mu/2})(p+i)^{-1}\jap{p}^{\mu/2}\\
& &\int_\C\frac{\partial \phi^\C}{\partial \bar{z}}(z-p)^{-1}qV(z-p)^{-1}dz\wedge d\bar{z}\jap{p}^{\mu/2}S_1
\end{eqnarray*}
where $\phi^\C$ is an almost analytic extension of $u$.\\ 
 We can remark that, since $\mu<2$ , $(p+i)^{-1}\jap{p}^{\mu/2}$ is bounded with bound less than $1$.
Denote
\[I=\frac{1}{2\pi}\int_\C\frac{\partial \phi^\C}{\partial \bar{z}}(z-p)^{-1}qV(z-p)^{-1}dz\wedge d\bar{z}.\]
Since $|q|V$ is bounded, we have
\begin{eqnarray*}
\left\|\jap{p}^{\mu/2}I\right\|&\leq&\frac{1}{\pi}\int\left|\frac{\partial \phi^\C}{\partial \bar{z}}\right|\frac{\jap{x}^{\mu/2}}{|y|}\|qV\||y|^{-1}dx\wedge dy\\
&\leq& C_1\int_{x\in\R}\int_{|y|\leq C_2\jap{x}}\jap{x}^{-(2+\mu/2)}dx\wedge dy\\
&\leq& C_3,
\end{eqnarray*}
and similarly for $\|I\jap{p}^{\mu/2}\|$. Remark that this bound depends on $\|qV\|_\infty$.

Moreover, since $|q|^2V$ is bounded, by Hardy inequality, for all $f\in\cd(S^{1/2})$
\[
\left|(f,[qV,iu(p)]f)\right|\leq (1+\frac{2}{n-2})(\left\|\jap{p}^{\mu/2}I\right|+\left\|I\jap{p}^{\mu/2}\right|)\|S_1f\|^2\leq C\||q|^2V\|_\infty\|S_1f\|^2,
\]
where $C$ depends only of $\mu$ and $n$. 
For the second term of \eqref{eq:commut A_u}, 
\begin{eqnarray*}
V(div u)(p)&=&V\jap{p}^{-\mu}(n-\mu \Delta\jap{p}^{-2})\\
&=&\jap{p}^{-\mu/2}V(n-\mu \Delta\jap{p}^{-2})\jap{p}^{-\mu/2}\\
& &+[V,\jap{p}^{-\mu/2}](n-\mu \Delta\jap{p}^{-2})\jap{p}^{-\mu/2}.
\end{eqnarray*}
By Hardy inequality, if $|q|^2V$ is bounded, 
\[\left|(f,\jap{p}^{-\mu/2}V(n-\mu \Delta\jap{p}^{-2})\jap{p}^{-\mu/2}f)\right|\leq \|n-\mu \Delta\jap{p}^{-2}\|\||q|^2V\|\|S_1 f\|^2.\]
As previously, we also have
\[\left|(f,[V,\jap{p}^{-\mu/2}](n-\mu \Delta\jap{p}^{-2})\jap{p}^{-\mu/2}f)\right|\leq C'\|n-\mu \Delta\jap{p}^{-2}\|\||q|V\|\|S_1 f\|^2,\]
where $C'$ depends only of $\mu$.
Therefore, as for $(div u)(p) V$, by \eqref{eq:commut A_u}, we have
\[\left|(f,[V,\rmi A_u]f)\right|\leq C(\||q|^2V\|_\infty+\||q|V\|_\infty)\|S_1 f\|^2.\]
In particular, if $\||q|^2V\|_\infty$ and $\||q|V\|_\infty$ are small enough, then $V_1$ satisfies \eqref{eq:Virial A_u}.

For the second order commutator, we have
\begin{eqnarray*}
[[V,\rmi A_u],\rmi A_u]&=&-VA_uA_u+2A_u VA_u-A_uA_u V\\
&=&-V(q\cdot u(p)-\frac{i}{2}(div u)(p))^2\\
& &+(u(p)\cdot q+\frac{i}{2}(div u)(p))V(q\cdot u(p)-\frac{i}{2}(div u)(p))\\
& &-(u(p)\cdot q+\frac{i}{2}(div u)(p))^2V\\
&=&[q^2V,u(p)]u(p)+u(p)[q^2V,u(p)]+B
\end{eqnarray*}
where $B$ depends only of the first and second order derivatives of $u$. As previously, by Helffer-Sj\"ostrand formula and Hardy inequality, we can see that 
\[|(f,[q^2V,u(p)]f)|\leq C''\||q|^3V\|\|S_1f\|^2.\]

For the term with $B$, as for
$V (div u)(p)$, we can show that 
\[|(f,Bf)|\leq C'''(\||q|^3V\|+\||q|^2V\|+\||q|V\|)\|S_1f\|^2.\]
Thus $V$ satisfies \eqref{eq:A_u ordre 2}.
\qed\end{proof}

\section{Other possible conjugate operators}\label{s:dim 1}

Let $k\in\N^*$, $k<n$. We can write $\R^n=\R^k\times \R^{n-k}$. Denote $(x,y)\in \R^n$ where $x\in\R^k$ and $y\in \R^{n-k}$.

Since we only have to assume that $S$ is injective, we can take a conjugate operator $A$ of the form $A=A_{k}\otimes \mathbb{1}_{\R^{n-k}}$ where $A_k$ is one of the previous conjugate operators ($A_D, A_F, A_u$) on  $L^2(\R^k)$. We obtain the same results, with similar proofs, that previously with the operator $q$ of multiplication by $(x,y)$ replaced by the operator $q_x$ of multiplication by $x$ and the gradient $\nabla$ replaced by the gradient $\nabla_x$ having only derivatives on $x$. For Theorem \ref{th:BoGo}, we get

\begin{theorem}
Let $k\geq3$.
Let $V_1,V_2\in L^1_{loc}(\R^n,\R)$ and $H=\Delta+V_1+\rmi V_2$.
Assume that
\begin{enumerate}[label=(\textbf{H.\arabic*}),ref=\textbf{(H.\arabic*)}]
\item $V_i$ are $\Delta$-bounded with bound less than one $1$ et $V_2\geq0$;

\item $\nabla_x V_i$ and $q_x\nabla_x V_i$ are $\Delta$-bounded and 
\[\left|(x\nabla_x)^2 V_i(x,y)\right|\leq \frac{C}{|x|^2}\quad \forall x\in\R^k\backslash\{0\};\]

\item There is $C'\in [0,(k-2)^2/2)$ such that
\begin{eqnarray*}
x\nabla_x V_1(x,y)&\leq&\frac{C'}{|x|^2};
\end{eqnarray*}
\end{enumerate}
Then Theorem \ref{th: Theorem B.1} applies and
\[\sup_{\lambda\in\R,\eta>0}\||q_x|^{-1}(H-\lambda+\rmi \eta)^{-1}|q_x|^{-1}\|<\infty.\]

Moreover, $H$ does not have eigenvalue in $\R$.
\end{theorem}
Remark that, as in Sections \ref{s:A_F} and \ref{s:A_u}, since we can not obtain the laplacian in the expression of the first order commutator between $H$ and $A$, we can not use a constant $c_1>0$. 

For Theorem \ref{th:sr cond}, we have the following

\begin{theorem}
Let $k\geq 3$.
Let $V_1,V_2\in L^1_{loc}(\R^n,\R)$ and $H=\Delta+V_1+\rmi V_2$ with $V_2\geq0$.
Assume that $|q_x|^2V_1$ and $|q_x|V_1$ are bounded with bound small enough and that $\jap{q_x}^3V_i$ is bounded.
Then Theorem \ref{th: Theorem B.1} applies. In particular, for all $1\leq\mu<2$,
\[\sup_{\rho\in\R,\eta>0}\|\jap{p_x}^{-\mu/2}|q_x|^{-1}(H-\rho+\rmi \eta)^{-1}|q_x|^{-1}\jap{p_x}^{-\mu/2}\|<\infty.\]

Moreover, $H$ does not have eigenvalue in $\R$.
\end{theorem}

This choice of conjugate operator permits to consider potentials $V$ of the form $V=W\otimes W'$ where $W:\R^k\rightarrow\R$ has good properties of decrease and of regularity and  $W':\R^{n-k}\rightarrow\R$ is bounded with a bound small enough. In particular, no conditions of decrease or on the derivatives of $W'$ are imposed.

\section{An oscillating potential}\label{s:oscillant}
In this section, we will see what conditions our different results impose on an oscillating potential to use Theorem \ref{th: Theorem B.1} with this potential for $V_1$.

Let $n\geq3$, $\alpha>0$, $\beta>0$, $k,w\in\R^*$ and $\kappa\in C^\infty_c(\R,\R)$ such that $\kappa=1$ on $[-1,1]$ and $0\leq \kappa\leq 1$.
Let \begin{equation}\label{eq:oscillant}
W_{\alpha \beta}(x)=w(1-\kappa(|x|))\frac{\sin(k|x|^\alpha)}{| x|^\beta}.
\end{equation}
Notice that this potential was already studied in \cite{BaD, DMR, DR1, DR2, JMb, Ma1, ReT1, ReT2} but the limiting absorption principle was only proved for high energy, far from the threshold zero.

Remark that $W_{\alpha \beta}$ does not satisfy assumptions of Theorem \ref{th: MaRi}. In fact, because of the oscillations, $-x\cdot\nabla W_{\alpha \beta}$ is not positive. Moreover, if $\alpha>0$, $x\cdot\nabla W_{\alpha \beta}$ can be unbounded.

By a simple calculus, we can see that 
\begin{equation}\label{eq:commut oscillant}
x\cdot \nabla W_{\alpha \beta}(x)=-w\kappa'(|x|)\frac{\sin(k|x|^\alpha)}{| x|^{\beta-1}}-\beta W_{\alpha \beta}(x)+kw\alpha(1-\kappa(|x|))\frac{\cos(k|x|^\alpha)}{| x|^{\beta-\alpha}}.
\end{equation}

We have the following
\begin{theorem}\label{th:oscillant}
Let $W_{\alpha \beta}$ as above.
\begin{itemize}
\item If $\beta\geq2$ and $\beta-2\alpha\geq2$, then, for $w$ small enough, Theorem \ref{th:BoGo} applies with $V_1=W_{\alpha \beta}$ and $V_2=0$. 

\smallskip

\item If $\beta\geq2$ and $\beta-\alpha\geq2$, then, for $w$ small enough, Theorem \ref{th:A_F 2} applies with $V_1=W_{\alpha \beta}$ and $V_2=0$. 

\smallskip

\item If $\beta\geq3$, then, for $w$ small enough, Theorem \ref{th:sr cond} applies with $V_1=W_{\alpha \beta}$ and $V_2=0$. 
\end{itemize}
\end{theorem}
 We make some remarks about this result:
 \begin{enumerate}
\item Remark that the condition $\beta\geq2$ and $\beta-\alpha\geq2$ is satisfied if we assume $\beta\geq2$ and $\beta-2\alpha\geq2$.

\smallskip
 
 \item For this type of potential, the Limiting Absorption Principle was already proved on all compact subset of $(0,+\infty)$ (see \cite{JMb,Ma1}). Moreover, by the Lieb-Thirring inequality, we already know that, for $w$ small enough, there is no negative eigenvalue. Here, we show that, moreover, a Limiting Absorption Principle can be proved on all $\R$ and thus that zero is not an eigenvalue.
 
 \smallskip
 
 \item We always assume that $V_2=0$ but we can choose 
 \[V_2(x)=(1-\kappa(|x|))\frac{\sin(k|x|^\gamma)+1}{| x|^\delta}\] with similar conditions on $\gamma,\delta$.
 
 \smallskip
 
 \item If we want to use Theorem \ref{th:sr cond}, remark that no conditions are impose on $\alpha$. In particular, $W_{\alpha \beta}$ can have high oscillations at infinity and we can replace $|x|^\alpha$ by $e^{|x|^2}$ or another function with the same conditions on $\beta$.
 
 \smallskip
 
 \item As it was explain in section \ref{s:dim 1}, if we write $\R^n=\R^k\times\R^{n-k}$, with $k\geq3$, we have the same conditions if $V_1(x,y)=W_{\alpha \beta}(x) W(y)$ for all $(x,y)\in\R^k\times\R^{n-k}$ with $W$ bounded and $\Delta$-compact.
 \end{enumerate}

 \begin{proof}[Theorem \ref{th:oscillant}]
 By \eqref{eq:commut oscillant}, we can see that if $\beta\geq2$ and $\beta-\alpha\geq2$, then $q\cdot\nabla V_1\geq\frac{c}{|x|^2}$, with $c\geq0$ small enough if $w$ is small enough. 
 
 \begin{itemize}
 
 \item By a simple computation, we can remark that $(q\cdot\nabla)^2V_1(x)=B_1(x)-k^2\alpha^2(1-\kappa(|x|))\frac{\sin(k|x|^\alpha)}{| x|^{\beta-2\alpha}}$ where $|q|^2B_1$ is bounded if $\beta\geq2$ and $\beta-\alpha\geq2$. Therefore, if $\beta-2\alpha\geq 2$, Theorem \ref{th:BoGo} applies.
 
 \smallskip

\item Remark that we already prove that if $\beta\geq2$ and $\beta-\alpha\geq2$, then $q\cdot\nabla V_1\geq\frac{c}{|x|^2}$, with $c\geq0$. Therefore, Theorem \ref{th:A_F 2} applies.

\smallskip

\item If $\beta\geq3$, then $|q|^3W_{\alpha \beta}$ is bounded. Moreover, if $w$ is small enough, $\||q|^2W_{\alpha \beta}\|_{\infty}$ is small enough. Thus Theorem \ref{th:sr cond}.\qed
\end{itemize}
 \end{proof}

\appendix

\section{The Helffer-Sj\"ostrand formula}\label{s: H-S formula}
Let $ad^1_B(T)=[T,B]$ be the commutator. We denote $ad^p_B(T)=[ad^{p-1}_B(T),B]$ the iterated commutator. Furthermore, if $T$ is bounded, $T$ is of class $C^k(B)$ if and only if for all $0\leq p\leq k$, $ad^p_B(T)$ is bounded.

\begin{proposition}[\cite{DG} and \cite{Mo}]
Let $\varphi\in \cs^\rho$, $\rho\in\R$. For all $l\in\R$, there is a smooth function $\varphi^\C:\C\rightarrow\C$, called an almost-analytic extension of $\varphi$, such that :
\begin{equation}
\varphi^\C_{|\R}=\varphi \qquad \frac{\partial \varphi^\C}{\partial \bar{z}}=c_1 \langle\Re(z)\rangle^{\rho-1-l}|\Im(z)|^l
\end{equation}
\begin{equation}
supp \varphi^\C\subset \{x+iy\|y|\leq c_2\langle x\rangle\}
\end{equation}
\begin{equation}
\varphi^\C(x+iy)=0, \text{ if }x\notin supp(\varphi)
\end{equation}
for constant $c_1$ and $c_2$ depending of the semi-norms of $\varphi$.
\end{proposition}

\begin{theorem}[\cite{GoJ} and \cite{Mo}]
Let $k\in\N^*$ and $T$ a bounded operator in $C^k(B)$. Let $\rho<k$ and $\varphi\in\cs^\rho$. We have

\begin{equation}\label{eq: H-S formula}
[\varphi(B),T]=\sum^{k-1}_{j=1}\frac{1}{j!}\varphi^{(j)}(B)ad^j_B(T)+\frac{i}{2\pi}\int_\C \frac{\partial \varphi^\C}{\partial\bar{z}}(z-B)^{-k}ad^k_B(T)(z-B)^{-1} dz\wedge d\bar{z}
\end{equation}
\end{theorem}

In the general case, the rest of the previous expansion is difficult to calculate. So we will give an estimate of this rest.

\begin{proposition}[\cite{GoJ} and \cite{Mo}]\label{prop: estimate H-S}
Let $T\in C^k(A)$ be a self-adjoint and bounded operator. Let $\varphi\in\cs^\rho$ with $\rho<k$. Let 
\[I_k(\varphi)=\int_\C \frac{\partial \varphi^\C}{\partial\bar{z}}(z-B)^{-k}ad^k_B(T)(z-B)^{-1} dz\wedge d\bar{z}\]
 be the rest of the development of order $k$ in \eqref{eq: H-S formula}. Let $s,s'>0$ such that $s'<1$, $s<k$ and $\rho+s+s'<k$. Then $\langle B \rangle^s I_k(\varphi)\langle B\rangle^{s'}$ is bounded.
\end{proposition}

In particular, if $\rho<0$, and if we choose $s'$ near $0$, we have $\langle B \rangle^s I_k(\varphi)\langle B\rangle^{s'}$ bounded, for all $s<k-s'-\rho$.

{\bf Acknowledgements.} I thank my doctoral supervisor, Thierry Jecko, for fruitfull discussions and comments. 

\bibliographystyle{alpha}
\bibliography{bibliographie}

\end{document}